\documentclass[12pt, singlespace]{amsart}
\usepackage{amsmath,amssymb,amsthm}
\usepackage[dvipdfm]{graphicx}
\usepackage{enumerate}
\usepackage[dvips]{color}
\usepackage{version}

\newtheorem{Thm}{Theorem}[section]
\newtheorem{Lem}[Thm]{Lemma}

\newtheorem{Cor}[Thm]{Corollary}

\newtheorem{Prop}[Thm]{Proposition}
\newtheorem{Conj}[Thm]{Conjecture}
\newtheorem{``Conj"}[Thm]{``Conjecture"}

\newtheorem{Claim}[Thm]{Claim}

\theoremstyle{remark}

\theoremstyle{definition}
\newtheorem{Def}[Thm]{Definition}

\newtheorem*{ack}{Acknowledgments}

\newcommand{\lct}{\mathop{\mathrm{lct}}\nolimits}

\newcommand{\Aut}{\mathop{\mathrm{Aut}}\nolimits}

\newcommand{\SL}{\mathop{\mathrm{SL}}\nolimits}
\newcommand{\GL}{\mathop{\mathrm{GL}}\nolimits}
\newcommand{\PGL}{\mathop{\mathrm{PGL}}\nolimits}

\newcommand{\DF}{\mathop{\mathrm{DF}}\nolimits}

\excludeversion{memo}

\begin{document}

\title{On the moduli of K\"ahler-Einstein Fano manifolds}
\author{Yuji Odaka}
\address{Department of Mathematics, Kyoto university, Oiwake-cho, 
Kitashirakawa, Sakyo-ku, Kyoto city, Kyoto, 606-8285, JAPAN}
\email{yodaka@math.kyoto-u.ac.jp}

\maketitle
\thispagestyle{empty}

\begin{abstract}
We prove that 
K\"ahler-Einstein Fano manifolds with finite automorphism groups 
form Hausdorff moduli algebraic space with only quotient singularities. 
We also discuss the limits as $\mathbb{Q}$-Fano varieties 
which should be put on the boundary of its canonical compactification. 
\end{abstract}

\section{Introduction}

After the well-known construction of moduli of stable curves, 
recently the construction of 
moduli of canonically polarized varieties (thus of general type) of higher dimension 
is established by the works of Shepherd-Barron, Koll\'ar, Viehweg, Alexeev and 
subsequent authors, which depends on birational geometric techniques after the 
minimal model program, admitting \textit{semi-log-canonical} singularities. 

On the other hand, 
there are no general theory of constructing moduli of 
anti-canonically polarized varieties (i.e. of Fano type) nor of more general 
polarized varieties. 
In this exposition, we discuss a canonical way of constructing moduli 
using some differential geometry. That is, we will show the proof of the following theorem. 

\begin{Thm}\label{mod.sp}
K\"ahler-Einstein Fano $n$-folds with finite automorphism groups 
form Hausdorff moduli algebraic space in the sense of Artin \cite{Art71}, 
which is actually a (complex) orbifold. 
\end{Thm}

As originally expected, the existence of K\"ahler-Einstein metrics on a Fano manifold is proved to be equivalent to the K-stability. The proof of stability of KE Fano manifolds is established by \cite{Mab08}, \cite{Mab09}, \cite{Ber12} and for the converse direction, please consult \cite{CDS12a}, \cite{CDS12b}, \cite{CDS13}, \cite{Tia12b} which discuss after the Donaldson's proposal of a new continuity method using conical singular K\"ahler-Einstein metrics \cite{Don11}. Thanks to the equivalence, we can paraphrase 
as the following. 

\begin{Thm}[Re-statement]\label{mod.sp.K}
Smooth K-stable Fano varieties form separated moduli algebraic space 
and moreover we have only quotient singularities on it. 
\end{Thm}

\noindent 
At the first talk at Kinosaki symposium (2010), 
the author suggested more general conjecture (K-moduli) - to construct 
moduli of K-(poly)stable varieties - which we review and discuss at the final section. So in a sense, this is a progress report on the conjecture. 

This principle of constructing moduli of these varieties with canonical metrics has been  often hided but existed behind the whole history of moduli construction in a sense. 
To see a bit of the connection, note that K\"ahler-Einstein metrics exist on smooth projective curves, 
smooth canonical models, and Calabi-Yau manifolds 
due to the Aubin-Yau existence theory (\cite{Aub78}, \cite{Yau78}) while 
we know the existence of quasi-projective moduli of them by algebro-geometric independent work  \cite{Vie95}. 
What we discuss here is how to directly use K\"ahler-Einstein geometry for the construction of moduli as varieties or algebraic spaces. 

Personally speaking, 
the author's original motivation for this program of constructing moduli 
is from his algebro-geometric observations on the K-stability \cite{Od08}, 
\cite{Od11}, which connected 
geometrically compactified moduli variety of general type varieties of the Koll\'ar-Shepherd-Barron-Alexeev type (cf. e.g. \cite{Kol10}) and the K-stability \cite{Don02}. 
As originally expected, the existence of K\"ahler-Einstein metrics on a Fano manifold is proved to be equivalent to the K-stability. The proof of stability of KE Fano manifolds is established by \cite{Mab08}, \cite{Mab09}, \cite{Ber12} and for the converse direction, please consult \cite{CDS12a}, \cite{CDS12b}, \cite{CDS13}, \cite{Tia12b} which discuss after the Donaldson's proposal of a new continuity method using conical singular K\"ahler-Einstein metrics \cite{Don11}. 

While the equivalence is now established, the K-polystability (or the existence of K\"ahler-Einstein metric) for a given or a concrete Fano manifold, and the behaviour with respect to complex deformation have both certain independent aspects and they are difficult interesting questions in general which remain. To that question, we can prove 
the following. 

\begin{Thm}\label{M.thm}
For a flat projective family of Fano manifolds $\pi \colon \mathcal{X}\rightarrow S$,
then the subset: 

\begin{equation*}
\{ s\in S\mid \mathcal{X}_{s} \mbox{ is a KE Fano manifold with } \#\Aut(X)<\infty \} \subset S
\end{equation*} is a Zariski open subset of $S$. 
\end{Thm}

\noindent
Note that this Zariski openness question is also discussed in 
Donaldson's informal document \cite{Don09} 
which takes a different approach under several conjectural hypotheses. 
Indeed, some discussion with Donaldson on 
this question is also reflected in this notes as we will explain at each point. 
A fundamental remark is that this Zariski openness does \textit{not} hold without the  condition of discreteness of automorphism group. 

Coming back to the moduli construction, 
the remained problem after Theorem \ref{mod.sp}, \ref{mod.sp.K} 
is to construct the compactification of them by attaching 
\textit{Gromov-Hausdorff limits}. Indeed the following result established 
those algebraicity. 

\begin{Thm}\label{DS12}
Gromov-Hausdorff limit of K\"ahler-Einstein Fano $n$-folds is a KE 
$\mathbb{Q}$-Fano variety, i.e. log-terminal Fano varieties with a singular K\"ahler-Einstein metric. 
\end{Thm}

\noindent 
Note that the meaning of Gromov-Hausdorff convergence is enhanced there 
to encode the continuity of complex structures. See \cite{DS12} for the details. 

In the case of Del Pezzo surfaces, the above algebraicity was established by Tian two decades ago \cite{Tia90}. In that case, we succeeded in the construction of the compactification in explicit way in the joint work with Cristiano Spotti and Song Sun 
\cite{OSS12}. This is pioneered by Mabuchi-Mukai \cite{MM93} where degree $4$ case 
is written. 

\section{Proofs}

\begin{proof}[Proof of Theorem \ref{M.thm}]

To state the proof, we need to recall the following characterization 
of test configurations and introduce a certain modification of the K-stability notion. 

\begin{Prop}[{\cite[Proposition 3.7]{RT07}}]\label{tc.1-ps}

Let $X$ be a Fano manifold. 
Then for a large enough $m\in \mathbb{Z}_{>0}$, 
a one-parameter subgroup of $\GL(H^{0}(X,\mathcal{O}_{X}(-mK_X)))$ is equivalent to the data of a test configuration 
$(\mathcal{X},\mathcal{L})$ of $(X,-K_X)$ 
whose polarization $\mathcal{L}$ is very ample  (over $\mathbb{A}^{1}$) 
with exponent $m$. 

\end{Prop}

\begin{Def}
A $\mathbb{Q}$-Fano variety $X$ is said to be $K_{m}$-stable (resp., 
$K_m$-semistable) if and only if 
$\DF(\mathcal{X},\mathcal{L})>0$ for any test configuration of the above type (i.e., 
with very ample $\mathcal{L}$ of exponent $m$) which is 
non-trivial test configuration whose central fiber is normal. 
\end{Def}

Notice that what \cite{LX11} proved can be rephrased as 
that a $\mathbb{Q}$-Fano variety is 
K-stable (resp., K-polystable, K-semistable) if and only if it is 
$K_{m}$-stable (resp., $K_m$-polystable, $K_m$-semistable) 
for sufficiently divisible $m\in \mathbb{Z}_{>0}$. 
The name ``$K_m$-stability" followed a suggestion of 
Professor Simon Donaldson in a discussion. 
The equivalent notion was also treated in \cite{Tia12a}.  

Let us also recall  the definition of the algebro-geometric counterpart of the  $\alpha$-invariant (\cite{Tia87}) i.e., the so-called \textit{global log-canonical threshold}  (\cite{Sho92}) and its variants as follows. 

\begin{Def}
Suppose $X$ is a projective manifold, $E$ is an effective $\mathbb{Q}$-divisor on $X$, 
and $L$ is an ample line bundle on $X$. Then we consider the following invariants. 

$${\rm lct}(X,E):={\rm sup}\{ \alpha>0 \mid (X,\alpha E) \mbox{ is log canonical} \}, $$

$${\rm lct}_l(X;L):={\rm inf}{_{D\in |lL|}}
{\rm lct}\big(X, \frac{1}{l}D \big) \mbox{ for } l \in \mathbb{Z}_{>0}, $$

and 
$${\rm glct}(X;L):={\rm inf}{_{l\in\mathbb{Z}_{>0}}}
{\rm lct}_l(X;L), $$
which are called the \textit{log canonical threshold} of $X$ with respect to $E$, 
the \textit{log canonical threshold} of $X$ with respect to $L$ with exponent $l$, 
and the \textit{global log canonical threshold}  of $X$ with respect to $L$. 

For a Fano manifold $X$, we simply write 
${\rm glct}(X):={\rm glct}(X;-K_X)$, $\lct_l(X):=\lct_l(X;-K_X)$. 
\end{Def}

According to the approximation theory of K\"ahler metrics by Bergman metrics, 
the last notion corresponds to $\alpha$-invariant (cf.\ \cite{Tia87}, 
\cite[Appendix]{CSD08}). 

We use the following general proposition in what follows. 

\begin{Prop}\label{lower.bounds}
Given a flat family of Fano manifolds $\mathcal{X}\rightarrow S$ and its 
relative pluri-anticanonical divisors $\mathcal{D}\in |-\lambda K_{\mathcal{X}/S}|$, 
we have uniform positive lower bounds for both $\lct(\mathcal{X}_s;-K_{\mathcal{X}_s})$ 
and $\lct(\mathcal{D}_s;-K_{\mathcal{X}_s}|_{\mathcal{D}_s})$ 
where $s$ runs through $S$. 
\end{Prop}
\begin{proof}[proof of Proposition \ref{lower.bounds}]
First, let us recall the following basic lemma. 
\begin{Lem}
For any log pair $(X,D)$ with $\mathbb{Q}$-Cartier effective $\mathbb{Q}$-divisor $D$,
$$\lct(X,D)\geq 1/{\it mult}_p(D),$$ 
where if $D$ is an effective $\mathbb{Q}$-divisor $\sum a_i D_i$ with prime divisors $D_i$ and $a_i\in \mathbb{Q}_{>0}$, ${\it mult}_p(D):=\sum a_i {\it mult}_p(D_i)$. 
\end{Lem} 
\noindent
This Lemma follows from 
the upper semicontinuity of the log canonical thresholds (cf. e.g., \cite{Mus02}), 
applied to a degenerating family of $D$ to a multiple of a smooth local divisor 
(cf. e.g., \cite{Kol97} for more details). 

On the other hand, supposing $-\lambda K_{X}$ is very ample for 
every Fano $n$-fold $X$, we have 
$$
{\it mult}_p(D)\leq \lambda^{n}(-K_X)^{n} 
$$
for all $p\in D$ with $D\in |-\lambda K_X|$. 
Indeed, if we take general members $D_1, \cdots, D_{n-1}$ among 
$|-\lambda K_X|$ which pass through the point $p\in X$ without singularity at $p$, 
the inequality follows from the standard estimate of the intersection number 
$(D_1.\dots.D_{n-1}.D)=\lambda^{n}(-K_X)^{n}$. 
The proof of Proposition \ref{lower.bounds} is completed. 
\end{proof}

On the other hand, let us recall the following results from \cite{OS11}. 

\begin{Thm}[{\cite[Theorem 4.1]{OS11}}]\label{CY.gen}

Let $X$ be a projective variety, $D'$ is an effective $\mathbb{Q}$ divisor and 
$0\leq \beta\leq 1$. 
{\rm (i)} Assume $(X, (1-\beta)D)$ is a log Calabi-Yau pair, i.e., $K_X+(1-\beta)D$ is numerically equivalent to  the zero divisor and it is a semi-log-canonical pair (resp.\ kawamata-log-terminal pair). Then, $((X,D),L)$ is logarithmically K-semistable (resp.\ logarithmically K-stable) with respect to the boundary's parameter $\beta$. 

{\rm (ii)} Assume $(X,(1-\beta)D')$ is a semi-log-canonical model, i.e., $K_X+(1-\beta)D'$ is ample and  it is a semi-log-canonical pair. Then, $((X,D),K_X+(1-\beta)D')$ and $\beta \in \mathbb{Q}_{>0}$ is log K-stable with respect to the boundary's parameter $\beta$. 

\end{Thm}

\begin{Thm}[{cf. \cite[Corollary 5.5]{OS11}}]\label{weak.alpha}
Let $X$ be an arbitrary $\mathbb{Q}$-Fano variety and 
$D$ is an integral pluri-anti-canonical Cartier divisor $D\in |-\lambda K_X|$ with some $\lambda\in \mathbb{Z}_{>0}$. 
Then $(X,-K_X)$ is logarithmically K-stable (resp.\ logarithmically K-semistable) for cone  angle $2\pi\beta$ with $\frac{\lambda-1}{\lambda}<\beta<
\bigl(\lambda-1+(\frac{n+1}{n}){\rm min}\{{\rm glct}(X;-K_X), {\rm glct}(D;-K_X|_D)\}\bigr)/\lambda$ 
(resp.\ $\frac{\lambda-1}{\lambda} \leq\beta\leq
\bigl(\lambda-1+(\frac{n+1}{n}){\rm min}\{{\rm glct}(X;-K_X), {\rm glct}(D;-K_X|_D)\}\bigr)/\lambda$). 
\end{Thm}

\noindent 

We make some remarks. 
Recall that Theorem \ref{CY.gen} (i) extends and algebraically recovers \cite[Theorem 1.1]{Sun11}, \cite[Theorem1.1]{Bre11} and \cite[Theorem 2]{JMR11}. In other words, 
at least for the case where $X$ is smooth, $D$ is a smooth integral divisor, 
the above is expected to follow from their results as well. 
Regarding Theorem \ref{weak.alpha}, 
the original version of \cite[Corollary 5.5]{OS11} treated only $\lambda=1$ case, 
but it is recently noticed that it can be simply extended to the above form.  
We also remark that the case if $X$ and $D$ are both smooth and $\lambda=1$, 
Theorem \ref{weak.alpha} also follows from 
\cite[Theorem 1.8]{Ber10} combined with \cite[Theorem1.1]{Ber12}. 

The above two theorems imply that 

\begin{Cor}\label{w.weak.alpha}
For an arbitrary smooth Fano manifold $X$ and a smooth pluri-anticanonical 
divisor $D\in |-\lambda K_X|$, $(X,-K_X)$ is logarithmically K-stable (resp.\ logarithmically K-semistable) for cone  angle $2\pi\beta$ with $0<\beta<
(\lambda-1+(\frac{n+1}{n}){\rm min}\{{\rm glct}(X;-K_X), {\rm glct}(D;-K_X|_D)\})/\lambda$ 
(resp.\ $0 \leq\beta\leq
(\lambda-1+(\frac{n+1}{n}){\rm min}\{{\rm glct}(X;-K_X), {\rm glct}(D;-K_X|_D)\})/\lambda$). 
\end{Cor}

\noindent
Combining the above lower boundedness of log canonical thresholds and 
Corollary \ref{w.weak.alpha}, 
we can take uniform $\beta_{\infty}(>0)$ 
such that every Fano manifolds $\mathcal{X}_s$ 
($s\in S$) are log K-stable for cone angle $2\pi\beta_{\infty}$ with respect to the divisor Cartier $\mathcal{D}_s$. 

We use the following theorem for the proof as well. 
\begin{Thm}{\cite[Theorem 1]{CDS12b}, 
\cite[Theorem 2, the paragraph below it]{CDS13}}\label{log.GH}
For a Fano manifold $X$, 
if $X$ does not admit K\"ahler-Einstein metric, the following holds. 
For a smooth divisor $D\in |-\lambda K_{X}|$ with some $\lambda \in \mathbb{Z}_{>0}$, 
we have a non-trivial normal log test configuration $(\mathcal{X},\mathcal{D})$ 
of $(X,(1-\beta_{\infty})D)$ with cone 
angle $2\pi\beta_\infty$ ($1-\frac{1}{\lambda}<\beta_{\infty}\leq1$) such that $\DF_{\beta_{\infty}}((\mathcal{X},\mathcal{D});-K_{\mathcal{X}})=0$ and $\mathcal{X}_0$ is a $\mathbb{Q}$-Fano variety. Moreover, the $\mathbb{Q}$-Gorenstein index of $\mathcal{X}_{0}$ can be taken uniformly if we have a fixed dimension $\dim(X)$, $\lambda$ and 
a lower bound for $\beta_{\infty}$ which is bigger than $1-\frac{1}{\lambda}$. 
\end{Thm}

\noindent
Note that the bound of index of $\mathcal{X}_0$ is equivalent to 
that they are all parametrized in a fixed Hilbert scheme, by the recent boundedness result \cite[Corollary 1.8]{HMX12}. Note that the analogical 
bound of $\mathbb{Q}$-Gorenstein index of the limit as Theorem \ref{lower.bounds} 
was indeed already proved earlier in non-log setting  as well (\cite{DS12}). 

Thanks to the existence of the uniform $\beta_{\infty}$, 
in turn Theorem \ref{log.GH} implies that 
there is a large enough but fixed $m$ such that for any $s\in S$, $\mathcal{X}_s$ is K\"ahler-Einstein Fano manifold 
if and only if it is $K_m$-stable. 
Take a Hilbert scheme $H$ which exhausts all $\mathcal{X}_s$, embedded by 
$|-mK_{\mathcal{X}_s}|$. Automatically we have CM line bundle on $H$ which we denote by  $\Lambda_{CM}$ and there are actions of $\SL(h^0(-mK_{\mathcal{X}_s}))$ on both $H$ 
and $\Lambda_{CM}$ (linearization). 
Denote the locus in $H$ which parametrizes K\"ahler-Einstein Fano manifold  $\mathcal{X}_s$ with $s\in S$ by $H_{KE}$, and the locus which parametrizes 
normal fibers by $H_{normal}$. It is well known that $H_{normal}$ is a Zariski open 
subset of $H$ as a general phenomenon. 

Let us fix a maximal torus $T$ of $\SL(h^0(-mK_{\mathcal{X}_s}))$. 
Recall the following known fact for usual setting in Geometric Invariant 
Theory \cite{Mum65} which is for \textit{ample} linearized line bundles. 

\begin{Lem}[{cf.\ \cite[Chapter 2, Proposition 2.14]{Mum65}}]\label{PL.mum}
Consider the action of a reductive algebraic group $G$ on a polarized projective 
scheme $(H,\Lambda)$ i.e. $H$ is a projective scheme and $\Lambda$ is an ample 
linearized line bundle. Let us denote the corresponding GIT weights function with respect to a one parameter subgroup $\varphi\colon \mathbb{G}_m\rightarrow G$ at $h\in H$  by $\lambda(h,\varphi;\Lambda)$. 
Then there are some finite linear functions $\{l_i\}_{i\in I}$ on ${\rm Hom}_{{\rm alg.grp}}(\mathbb{G}_m,T)\otimes_{\mathbb{Z}} \mathbb{R}$ with rational coefficients, indexed by 
a finite set $I$ such that 
the weight function $\lambda(h,-;\Lambda)$, regarded as a function from 
${\rm Hom}_{{\rm alg.grp}}(\mathbb{G}_m,T)$ to $\mathbb{Z}$, extends to a 
piecewise linear rational function of the form 
$$\sup\{l_j(-)\mid j\in J_h\}$$ with some $J_h\subset I$. 
Moreover, $\psi\colon H\rightarrow 2^{I}$ which maps $h$ to $J_h$, 
is constructible in the sense that for any $J\in I$, 
$\{h\in H \mid J_h=J\}$ is constructible. 
\end{Lem}

\noindent
The last line is not stated in \cite[Proposition 2.14]{Mum65} but follows rather easily from the proof and the arguments written there. 

Recall that the CM line bundle \cite{PT06} (cf. also \cite{FS90}) 
is a line bundle defined on the base scheme 
for an arbitrary polarized family and the formal GIT weight of the CM line bundle on 
the Hilbert scheme is exactly the corresponding Donaldson-Futaki invariant. 
However, unfortunately we can not directly apply to the CM line bundle $\Lambda_{CM}$ which is \textit{not} necessarily ample in general. 
Nevertheless, Lemma \ref{PL.mum} can be still extended to the case for our $H$ and the CM line bundle $\Lambda_{CM}$ on it, though $\lambda(t,-;\Lambda_{CM})$ is not necessarily convex. To see that, first we take a Pl\"ucker polarization $\Lambda_{Pl}$ on $H$ which is known to be very ample, yielding the Pl\"ucker embedding of Hilbert scheme. Take sufficiently large $l\in \mathbb{Z}_{>0}$ such that $\Lambda_{CM}\otimes \Lambda_{Pl}^{\otimes l}$ is also very ample. Then the point is that 
$\lambda(h;-;\Lambda_{CM})=\lambda(h;-;\Lambda_{CM}\otimes \Lambda_{Pl}^{\otimes l})
-l\mu(h;-;\Lambda_{Pl})$. Therefore Lemma \ref{PL.mum} extends to 
the case with our setting that Hilbert scheme is $H$, $L=\Lambda_{CM}$ and 
$G=\SL(h^0(-mK_{\mathcal{X}_s}))$ as follows. 

\begin{Lem}
There are some finite linear rational functions $\{l_i\}_{i\in I'}$ on ${\rm Hom}_{{\rm alg.grp}}(\mathbb{G}_m,T)\otimes_{\mathbb{Z}} \mathbb{R}$ with rational coefficients, indexed by 
a finite set $I'$ such that the followings hold. For each $[X\subset \mathbb{P}]$, 
there is $J_X\subset I'$ such that 
its Donaldson-Futaki invariants with respect to one parameter subgroups of 
$G:=\SL(h^0(-mK_{\mathcal{X}_s})$ regarded as a function from 
${\rm Hom}_{{\rm alg.grp}}(\mathbb{G}_m,T)$ to $\mathbb{Z}$, extends to a 
piecewise linear rational function whose pieces are $\{l_j\}_{j\in J_X}$ 
with some $J_X\subset I'$. 
Moreover, $\psi$ is constructible in the sense that for any subset $J_X\subset I'$, 
$\{[X\subset \mathbb{P}]\in H \mid J_X=J\}$ is constructible. 
\end{Lem}

Next, let us discuss from another viewpoint which is more birational-geometric. 
This is related to what Donaldson called ``splitting of orbits" in \cite{Don09}. 
Considering the family over $S$ and trivialize the locally free coherent sheaf  $\pi_*\mathcal{O}(\mathcal{D})$ on a Zariski open covering $\{S_i\}_i$ we have 
the corresponding morphism $(\sqcup S_i\times T) \rightarrow H$, which can be regarded as a 
rational map $(\sqcup S_i\times (\mathbb{P}^1)^r) \dashrightarrow H$. 
For simplicity, re-set $S:=\sqcup S_i$ from now on which does not lose generality 
from our assertion of the theorem. 
Thinking of a resolution of indeterminacy, we get some blow up of $S\times (\mathbb{P}^1)^r$ along a closed subscheme inside $S\times ((\mathbb{P}^1)^r\setminus T)$. The blow up is generically isomorphic over $S\times ((\mathbb{P}^1)^r\setminus T)$, 
which means the GIT weights with respect to any fixed one parameter subgroups in $T$ 
stay constant for some Zariski open dense subset $S'$ of $S$. 
Replace $S$ by $S\setminus S'$ then we can 
construct a constructible stratification of $S$ such that each stratum's 
degenerations with respect to any fixed one parameter subgroup of $G$ fits into a 
flat family again. This implies that those corresponding Donaldson-Futaki invariants are 
constant on each strata by \cite{PT06}. We can also prove it by using Wang's interpretation of Donaldson-Futaki invariant \cite{Wan12}. 

\begin{Claim}\label{constructible}
There is a finite set $I$ of one parameter subgroups $\{\varphi_{i}\}\colon 
\mathbb{G}_m\rightarrow T$ ($i\in I$) such that the following is true. 
There is a stratification of $H_{normal}$ by constructible subsets $\{H_J\}_{J\subset I}$, 
such that for any $t\in H_J \subset H_{normal}$, $\mathcal{X}_t$ is $K_m$-stable 
(resp., $K_m$-semistable) if and only if 
$\DF(\varphi_j; [\mathcal{X}_t\subset \mathbb{P}])>0$ (resp., $\DF(\varphi_j; [\mathcal{X}_t\subset \mathbb{P}])\geq 0$) for all $j\in J$ where $\DF(\varphi_j;[\mathcal{X}_t\subset \mathbb{P}])$ actually only depends on $j$ and $J$. Here $\DF(\varphi_j;[\mathcal{X}_t\subset \mathbb{P}])$ 
means the Donaldson-Futaki invariant of the test configuration 
of $(\mathcal{X}_t,\mathcal{O}(1)=\mathcal{O}(-mK_{\mathcal{X}_t}))$ 
associated to the one parameter subgroup via Proposition \ref{tc.1-ps}. 
\end{Claim}

Summarizing up, we proved that the locus of Fano manifolds which are $K_m$-stable (resp., 
$K_m$-semistable, $K_m$-unstable, not $K_m$-stable) 
with respect to the maximal torus $T$'s action are all constructible subsets. 
Then the constructibility of 
$K_m$-stable locus in $H$ follows from 
C. Chevalley's lemma (cf.\ \cite{Har77}) which states that the image of 
constructible set by an algebro-geometric morphism is again constructible. 
Indeed, we can apply it to the group action morphism $\SL(h^0(-mK_{\mathcal{X}_s})) \times H\rightarrow H$ and the constructible subset of $H$ 
which parametrizes Fano manifolds which are not $K_m$-stable with respect to the 
action of $T$. Then, it shows that the subset of $H$ which 
parametrizes Fano manifolds with discrete automophism groups which are 
not $K_m$-stable is again constructible. 

On the other hand, using the deformation theory of 
K\"ahler metrics with constant scalar curvature it was shown in \cite{LS94} that 
\begin{equation*}
\{ s\in S\mid \mathcal{X}_{s} \mbox{ is a KE Fano manifold with } \#\Aut(X)<\infty \} \subset S
\end{equation*} 
is an open subset with respect to the Euclidean topology. More precisely speaking, 
we need to take resolution of singularities of $S$ to apply \cite{LS94}. 
Thus the locus of K\"ahler-Einstein Fano manifolds with finite automorphism groups  should be a Zariski open subset. We complete the proof of the main theorem. 

\end{proof}

We prove Theorem \ref{mod.sp} using Theorem \ref{M.thm} as follows. 

\begin{proof}[proof of Theorem \ref{mod.sp} (using Theorem \ref{M.thm})]
Consider the Hilbert scheme $H$ which exhausts all $m$-pluri-anticanonically embedded  Fano $n$-folds with uniform $m\in\mathbb{Z}_{>0}$. 
By applying the previous statement of Theorem \ref{M.thm} to 
a Zariski open covering of universal family over $H$, we obtain that 
the locus which parametrizes K\"ahler-Einstein Fano manifolds form 
a Zariski open locus. Therefore the natural moduli functor 
$\mathcal{F} \colon (Sch_{/\mathbb{C}})^{op}\rightarrow (Sets)$ whose $B$-valued  points set is defined by 
$$\{ \mbox{flat proj. } f\colon \mathcal{Y}\rightarrow B \mid  [\mathcal{Y}_{b}\subset \mathbb{P}(H^0(-mK_{\mathcal{Y}_b}))] 
\in H_{KE} \mbox{ for } \forall b\in B\}/\cong$$ 
naturally gives rise to 
Deligne-Mumford stack which is simply a quotient stack 
$[H_{KE}/\PGL(h^0(-mK_{\mathcal{X}_s}))]$. 
Then the Keel-Mori theorem \cite[Corollary 1.3]{KM97} implies that we have a coarse moduli algebraic space for it. 
The separatedness follows from a standard argument 
using the uniqueness of the limits as Riemannian manifolds (cf. e.g. \cite[Lemma 8.1]{CS10}) and the Bando-Mabuchi theorem \cite{BM85}, or \cite{Don01}. 

Moreover, the moduli algebraic stack (in the sense of Deligne-Mumford) is smooth as the 
deformation is unobstructed. It is because of the vanishing of obstruction group  $H^2(X,T_X)=H^2(X,\Omega^{n-1}\otimes K_{X}^{-1})=0$ thanks to the Akizuki-Nakano's generalization of the Kodaira vanishing \cite{AN54} combined with the finiteness of isotropy groups. Therefore the coarse moduli algebraic space has only quotient singularities. This completes the proof of Theorem \ref{mod.sp}. 

\end{proof}

\section{Revisiting the general conjecture}

The moduli space constructed in Theorem \ref{mod.sp} is expected to have 
canonical projective compactification by attaching (singular) $\mathbb{Q}$-Fano varieties on the boundary (cf.\ \cite{OSS12}) but let us review this in broader context 
as in \cite{Od10}. 
The author expects that K-semistable polarized varieties 
have coarse projective moduli in general. Including what we have already mentioned, the inspirations are mainly from the following backgrounds. 
\begin{itemize}
\item The \textit{moment map picture} on the infinite-dimensional 
space of complex structures by Fujiki, Donaldson (\cite{Fuj90}, \cite{Don97}, \cite{Don04}).  

\item The observations (\cite{Od08}, \cite{Od11}) that 
the projective moduli of general type varieties with semi-log-canonical singularities, 
constructed by using the \textit{minimal model program} (cf. \cite{Kol10}) 
and some other projective moduli corresponds to K-stability. 

\item The algebraicity of \textit{Gromov-Hausdorff limits} of K\"ahler-Einstein Fano manifolds (\cite{DS12}) and the Gromov (pre)compactness theorem (cf.\ \cite{MM93}, \cite{Spo12}, \cite{OSS12}). 
\end{itemize}

Or we can also say that the Yau-Tian-Donaldson conjecture itself is 
another background, as GIT polystably is originally designed by Mumford \cite{Mum65} to make them  form projective moduli although the K-stability is \textit{not} quite a GIT stability notion in that original sense. 
As a preparation for making this expectation rigorous, 
we introduce the following terminology. 
Let us call a numerical equivalence class of ample line bundle, 
a \textit{weak polarization} and a pair of projective variety and weak polarization 
a \textit{weakly polarized variety}. 

Then the following is a modified version of the conjecture stated in the Kinosaki 
symposium 2010. 

\begin{Conj}[K-moduli cf. \cite{Od10}]\label{K-moduli}
The moduli functor of K-semistable weakly polarized varieties has proper coarse moduli algebraic space\footnote{With 
infinitely many connected components. } $\bar{M}_{cscK}$. 
It set-theoritically parametrizes the set of all smooth K-polystable weakly polarized varieties and certain equivalence classes of 
singular weakly polarized varieties with singular cscK metrics. 
The connected components of $\bar{M}_{cscK}$ are all projective schemes, 
because the CM line bundle descends to ample line bundles on them. 
\end{Conj}

\noindent
We refer to \cite{OSS12} for 
detailed discussions on our $\mathbb{Q}$-Fano varieties case 
which largely settled (smoothable) del Pezzo surfaces case explicitly. 
In that paper, we have constructed explicit moduli for all degrees 
as the Gromov-Hausdorff compactifications. 
Indeed, the author expects more generally the K-moduli is close to a ``\textit{Gromov-Hausdorff compactification}" in appropriately rescaled sense, of moduli of smooth cscK polarised varieties. 

The point of view with Gromov-Hausdorff limit had been missed in the original K-moduli expectation \cite{Od10} and inspired by differential geometric colleagues later (e.g. \cite{DS12}, \cite{Spo12}), which the author is truly greatful.  Especially the memorable collaboration 
with Spotti and Sun \cite{OSS12} makes the author more confident on 
the general conjecture above. Moreover, very recently the wonderful 
work of Berman-Guenancia \cite{BG13} -extending hyperbolic metrics on 
stable curves- appeared which constructed singular 
K\"ahler-Einstein (thus cscK) metrics on semi-log-canonical varieties with ample canonical classes. We note that allowing just the semi-log-canonical singularities should be more general phenomenon for all classes of K-semistable (weakly) polarized  varieties \cite{Od08}. 

\section{About the limits}

An interesting problem is to recover or to have better understanding of 
the Gromov-Hausdorff limits of punctured family of K-stable varieties or 
``maximally destabilizing" test configurations of K-unstable varieties 
via the minimal model program along the line of 
\cite{Od08}, \cite{LX11}. (The author expects that the 
destabilizing test configuration of non-slc varieties \cite{Od08} via MMP 
is ``maximal" (optimal) in some appropriate sense. ) 

The following characterization for special case is essentially obtained in 
\cite{OS12}, which the author was not aware for a while.

\begin{Thm}[CM minimizing for (special) Fano varieties]\label{CM.min.Fano}
Suppose $\pi\colon \mathcal{X}\rightarrow C\ni 0$ is a $\mathbb{Q}$-Gorenstein flat projective family of $n$-dimensional $\mathbb{Q}$-Fano varieties over a projective curve $C$ and the alpha invariant of the central fiber $\it glct(\mathcal{X}_0;-K_{\mathcal{X}_0})$ 
is at least (resp.\ bigger than) $\frac{n}{n+1}$. Regarding $\pi$ as 
a polarized family but putting anticanonical line bundle on $\mathcal{X}$, any other flat polarized family  $(\mathcal{X}',\mathcal{L}')\rightarrow C$ which is isomorphic to $\mathcal{X}\rightarrow C$ away from the central fiber, we have 
$$
{\it deg}(CM(\mathcal{X},\mathcal{-K_{\mathcal{X}/C}}))\leq (resp.\ <)
{\it deg}(CM(\mathcal{X}',\mathcal{L}')). 
$$
\end{Thm}

In general, the author expect that the - a priori differential geometric notion - the Gromov-Hausdorff limit of punctured K-stable family of 
$\mathbb{Q}$-Fano varieties is, algebraically, characterized as \textit{what minimizes the CM degree} as above. We also expect the construction or above characterization is to be able to 
prove via the minimal model program or differential geometrically K\"ahler-Ricci flow. An easier version may be 
trivial family of K-semistable Fano varieties case (which, at least in smooth case by 
\cite{CDS13}, \cite{Tia12b}, isotrivially degenerates to 
K-polystable Fano manifolds). 

We only sketch the proof of Theorem \ref{CM.min.Fano} as it is essentially the arguments of 
\cite{OS12}. 
First we take a resolution of indeterminancy 
of $\mathcal{X}\dashrightarrow\mathcal{X}'$ as blow up $\pi\colon \mathcal{B}\rightarrow\mathcal{X}$ where 
$K_{\mathcal{X}}^{\otimes{(-r)}}(-E)$ with the Cartier exceptional divisor $E$ is the pullback of 
$\mathcal{L}'^{\otimes{(-r)}}$. Then we can estimate 
$${\it CM}(\mathcal{X}',\mathcal{L}')-
{\it CM}(\mathcal{X},K_{\mathcal{X}}^{\otimes{(-r)}})=
CM(\mathcal{B},\pi^*K_{\mathcal{X}}^{\otimes{(-r)}}(-E))-CM(\mathcal{X},\mathcal{L}), $$
i.e., the difference of CM lines of $\mathcal{X}$ and $\mathcal{X}'$, in a completely 
similar way as in \cite{OS12} in which they did for the case $\mathcal{X}$ is 
$\mathcal{X}_0\times \mathbb{P}^1$ via the basic theory of discrepancy. 

Analogously, for Calabi-Yau varieties and general type varieties, 
if we apply the estimates of Donaldson-Futaki invariants at \cite{Od11} similarly to 
the differences of CM degrees of a family and its blow up, which we can obtain as a 
resolution of rational map $\mathcal{X}\dashrightarrow \mathcal{X}'$, 
we can (re)prove the following. The semi-log-canonical models (``general type") case was obtained by X. Wang - C. Xu earlier 
\cite{WX12}, so in that sense, our results are its generalization to 
different types of varieties. 

\begin{Thm}[Analogues]

\begin{enumerate}

\item (Calabi-Yau case) 
Suppose $\pi\colon \mathcal{X}\rightarrow C\ni 0$ is a $\mathbb{Q}$-Gorenstein flat projective family of $n$-dimensional semi-log-canonical models over a projective curve $C$. Regarding $\pi$ as 
a polarized family but putting canonical line bundle on $\mathcal{X}$, any other flat polarized family  $(\mathcal{X}',\mathcal{L}')\rightarrow C$ which is isomorphic to $\mathcal{X}\rightarrow C$ away from the central fiber, we have 
$$
{\it deg}(CM(\mathcal{X},\mathcal{-K_{\mathcal{X}/C}}))\leq (resp.\ <)
{\it deg}(CM(\mathcal{X}',\mathcal{L}')). 
$$

\item (Canonical model case: \cite[Theorem 6]{WX12})
Suppose $\pi\colon \mathcal{X}\rightarrow C\ni 0$ is a $\mathbb{Q}$-Gorenstein flat projective (``KSBA") family of $n$-dimensional semi-log-canonical models over a projective curve $C$. Regarding $\pi$ as 
a polarized family but putting canonical line bundle on $\mathcal{X}$, any other flat polarized family  $(\mathcal{X}',\mathcal{L}')\rightarrow C$ which is isomorphic to $\mathcal{X}\rightarrow C$ away from the central fiber, we have 
$$
{\it deg}(CM(\mathcal{X},\mathcal{-K_{\mathcal{X}/C}}))\leq (resp.\ <)
{\it deg}(CM(\mathcal{X}',\mathcal{L}')). 
$$

\end{enumerate}

\end{Thm}

\begin{Cor}[well known? e.g. \cite{Bou13}]
If a punctured family of polarized log terminal Calabi-Yau varieties 
can be completed with a log terminal CY filling, there is no other semi-log-canonical 
CY filling. 

In particular, log terminal polarized Calabi-Yau varieties have Hausdorff moduli. 
\end{Cor}

We end this short article by saying regretfully that the above CM minimizing result was once claimed with Richard Thomas for any K-(semi)stable limit (family) whose proof had a gap. We still hope a possibility to recover that in general situation.

\begin{ack}
The author appreciates the organizers of the Kinosaki symposium 2013, 
S.~Kuroda, K.~Yamada, K.~Yoshioka, for inviting him to give the opportunity. 

The author would like to thank Professor S.~Donaldson 
for useful discussions and warm encouragements. 
He also would like to thank 
Professors I.~Cheltsov, C.~Xu, V.~Tosatti, S.~Boucksom, 
Doctors S.~Song, C.~Spotti,  
for useful discussions and comments, 
Professor G.~Tian for kindly sending the pdf file of \cite{Tia12a}. 
\end{ack}

\end{document}